\RequirePackage{fix-cm}
\documentclass[smallextended]{svjour3}       
\smartqed  
\usepackage{graphicx}
\usepackage[utf8]{inputenc}

\usepackage{amssymb,amsmath,array}
\usepackage{color}

\newcommand{\NN}{\mathbb{N}}
\newcommand{\RR}{\mathbb{R}}

\renewcommand{\AA}{\mathfrak{A}}
\newcommand{\A}{\mathcal{A}}
\newcommand{\B}{\mathcal{B}}
\newcommand{\C}{\mathcal{C}}
\newcommand{\D}{\mathcal{D}}
\newcommand{\BB}{\mathbf{B}}
\newcommand{\NNN}{\mathbf{N}}
\newcommand{\E}{\mathcal{E}}
\newcommand{\F}{\mathcal{F}}
\newcommand{\G}{\mathcal{G}}
\newcommand{\U}{\mathcal{U}}
\renewcommand{\S}{\mathcal{S}}
\newcommand{\1}{\mathbf{1}}

\newcommand{\dd}{\mathrm{d}}
\newcommand{\core}{\mathbf{core}}
\newcommand{\conv}{\mathbf{conv}}
\newcommand{\atoms}{\mathbf{atoms}}
\newcommand{\algebra}{\mathbf{algebra}}
\newcommand{\chains}{\mathbf{chains}}

\newcommand{\red}{\textcolor{black}}

\renewcommand{\wp}{\mathcal{P}}

\begin{document}

\title{How to assess coherent beliefs: A comparison of different notions of coherence in Dempster-Shafer theory of evidence\thanks{Accepted for publication in the volume {\it Reflections on the Foundations of Probability and Statistics: Essays in Honor of Teddy Seidenfeld}. T. Augustin, F. Cozman, and G. Wheeler (Eds.). {\it Theory and Decision Library A}. Springer.
}
}

\titlerunning{How to assess coherent beliefs}        

\author{Davide Petturiti         \and
        Barbara Vantaggi 
}


\institute{D. Petturiti \at
              Dept. Economics, University of Perugia\\
              \email{davide.petturiti@unipg.it}           
           \and
           B. Vantaggi \at
              Dept. MEMOTEF, La Sapienza University of Rome\\
              \email{barbara.vantaggi@uniroma1.it}
}

\date{Received: date / Accepted: date}

\maketitle

\begin{abstract}
Stemming from de~Finetti's work on finitely additive coherent probabilities, the paradigm of coherence has been applied to many uncertainty calculi in order to remove structural restrictions on the domain of the assessment. Three possible approaches to coherence are available: coherence as a consistency notion, coherence as fair betting scheme, and coherence in terms of penalty criterion. Due to its intimate connection with (finitely additive) probability theory, Dempster-Shafer theory allows notions of coherence in all the forms recalled above, presenting evident similarities with probability theory. In this chapter we \red{present} a systematic study of such coherence notions showing their equivalence.
\keywords{Coherence \and Belief function \and Betting scheme \and Penalty criterion}
\end{abstract}

\section{Introduction}
De Finetti's notion of coherence for probability assessments \cite{definetti} is well-known for its role in ruling probabilistic evaluations on a family of events without any particular algebraic structure (e.g., an algebra or a partition). Indeed, when a probability assessment is defined on an arbitrary set of events the axioms of (finitely additive) probability are no more sufficient and must be replaced by a suitable stronger condition.

A first definition of probabilistic coherence is that requiring the existence of a finitely additive probability measure on an algebra, extending the given assessment. Such definition of coherence is a consistency notion and is merely ``syntactic''. Two other popular definitions of coherence for probabilities \cite{definetti}, are that based on the betting scheme and that based on the penalty criterion. These two notions of coherence, mainly studied in decision theory, artificial intelligence and economics, provide also a ``semantic'' interpretation of the corresponding conditions. 

Coherence in terms of the penalty criterion has been extensively investigated by Teddy Seidenfeld (and colleagues) to whom we pay homage in this book in occasion of his 70th birthday (see, e.g., \cite{ssk-scoring,seidenfeld}).
In particular, in \cite{ssk-scoring} the authors carry out a systematic study of scoring rules for conditional probability assessments. Such paper investigates the connections with the betting criterion and also deals with scoring rules weakening the usual requirements of continuity and strict properness (see, e.g., \cite{scoring-rules}).

A key feature of probabilistic coherence is the possibility
of extending a coherent assessment to a new set of events by preserving coherence. In general, we have a class of coherent extensions, determining lower and upper envelopes that can be taken as non-additive measures (see, e.g., \cite{cs-libro,cs-bel,csv-is,pv-jmaa2018,walley-lp,walley-libro}). Hence, the notion of probabilistic coherence acts like a ``bridge'' between finitely additive probability theory and other uncertainty calculi. For instance, in some particular cases, due to specific
logical conditions, these envelopes
can be belief and plausibility functions \cite{cpv-is2016,csv-is,halpern} or necessity and possibility measures \cite{cpv-fss2016,dubprad-when}. A study in the case of countably additive measures has been carried out in \cite{krat}.

On the other hand, the paradigm of coherence has been extended substituting the class of finitely additive probability measures with another class of more general uncertainty measures (see, e.g., \cite{walley-lp,walley-libro,williams}). 
In particular, in \cite{ssk-imprecise} Seidenfeld (and colleagues) analyze scoring rules in the more general setting of imprecise probabilities.
Following this line, in this work we focus on Dempster-Shafer theory \cite{dempster,shafer} which is known to possess many important connections with finitely additive probability theory. We recall that the interplay between Dempster-Shafer theory and probability theory has remarkable consequences also in mathematical statistics \cite{cpv-sma,nguyen,ws-jspi,wasserman}.

We show that Dempster-Shafer theory allows notions of coherence in the three forms recalled for probabilities, i.e., coherence as a consistency notion, coherence as fair betting scheme, and coherence in terms of penalty criterion. In particular, the betting scheme notion is the quantitative counterpart of the qualitative notions introduced for preferences either on generalized lotteries \cite{cpv-kybernetika} or on gambles \cite{cpv-anor2019,cpv-ijar2019}.

All the definitions of coherence we introduce consider {\it partially resolving uncertainty}, as motivated by Jaffray in \cite{Jaffray-Bel}. More precisely, we assume that at the moment uncertainty is resolved, due to a lack of information, the agent can observe that an event $B \neq \emptyset$ has occurred, but he/she may not be able to identify the true state of the world. In detail, this translates in considering the gain (betting criterion) or the total penalty (penalty criterion) as a function defined on all non-impossible events, rather than a function defined on $\Omega$ (or on a partition of it). The latter situation is typical of the usual assumption of {\it completely resolving uncertainty}.  Further, we take for granted that the agent is systematically pessimistic as he/she always considers minima of indicators of events in computing such quantities.

We prove that all the introduced notions of coherence for belief assessments are equivalent and show that they reduce to the classical notions of coherence for probabilities when {\it completely resolving uncertainty} is assumed. Moreover, we specialize the introduced notions of coherence to work inside the subclass of finitely minitive necessity measures \cite{dubprad,shafer}, by assuming {\it consonance}, besides partially resolving uncertainty. Finally, the relation between coherence for belief functions and (strictly) proper scoring rules is studied and we show how the Bregman divergence \cite{cz-bregman,gs-sanfilippo,scoring-rules} determined by a bounded proper scoring rule can be used to correct an incoherent assessment.

The concept of coherence is of paramount importance for what concerns elicitation, interpretation and inference in an uncertainty framework, as shown by the theory of coherent lower previsions started by Williams \cite{williams} and Walley \cite{walley-lp,walley-libro}.
The purpose of the present chapter is twofold.
First, it provides a systematic treatment of coherence for partial assessments in the class of belief functions and in the two distinguished subclasses given by finitely additive probability measures and finitely minitive necessity measures. Second, by generalizing the classical betting and penalty criteria it gives operational tools for the elicitation and interpretation of belief and necessity assessments.

The chapter is structured as follows. In Section~\ref{sec:preliminaries} we collect the necessary preliminaries on belief functions. In Section~\ref{sec:coherence} we introduce the different notions of coherence for belief functions and prove their equivalence, while in Section~\ref{sec:specialcases} we investigate the notions of coherence for probability and necessity assessments. Finally, in Section~\ref{sec:properscoring} we present coherence conditions based on proper scoring rules and show their use in correcting an incoherent assessment.

\section{Preliminaries}
\label{sec:preliminaries}
Let $\Omega$ be an arbitrary non-empty set of {\it states of the world} and $\A \subseteq \wp(\Omega)$ an algebra of its subsets representing {\it events}, where $\wp(\Omega)$ stands for the power set of $\Omega$. In case of a finite algebra $\A$, let $\atoms(\A) = \{C_1,\ldots,C_m\}$ be the set of its {\it atoms}, that is the finest partition of $\Omega$ contained in $\A$. Moreover, for every $A \in \A$ we denote by $\1_A:\Omega \to \{0,1\}$ its {\it indicator}. Further, if $\E \subseteq \wp(\Omega)$ is an arbitrary non-empty family, then $\algebra(\E)$ denotes the minimal algebra of subsets of $\Omega$ containing $\E$.

A {\it belief function} (see \cite{dempster,shafer}) is a mapping $Bel:\A \to [0,1]$ such that:
\begin{description}
\item[\it (i)] $Bel(\emptyset) = 0$ and $Bel(\Omega) = 1$;
\item[\it (ii)] for every $k\geq 2$ and for every $A_1,\ldots,A_k \in \A$,
$$
Bel \left(\bigcup_{i=1}^k A_i\right) \geq \sum_{\emptyset \neq I \subseteq \{1,\ldots,k\}} (-1)^{|I| + 1} Bel\left(\bigcap_{i \in I} A_i\right).
$$
\end{description}

Condition {\it (ii)} above is usually termed {\it complete monotonicity}, moreover, together with condition {\it (i)} it implies {\it monotonicity} with respect to set inclusion, that is, $Bel(A) \le Bel(B)$ whenever $A \subseteq B$, for $A,B \in \A$. In short, in this chapter for belief function we mean a {\it normalized completely monotone capacity} \cite{choquet}.

Belief functions are intimately connected to finitely additive probabilities, in at least four aspects:
\begin{enumerate}
\item Every finitely additive probability measure $P$ on $\A$ is a belief function.
\item Every finitely additive probability measure $P$ on $\A$ determines a belief function on every super-algebra $\B$, with $\A \subset \B \subseteq \wp(\Omega)$, through its inner measure $P_*$ defined, for every $B \in \B$, as
$$
P_*(B) = \sup\{P(A)\,:\, A \subseteq B, A \in \A\}.
$$
More generally, it holds that the inner measure $Bel_*$ on $\B$ generated by a belief function $Bel$ on $\A$ is still a belief function \cite{deCooman-book}.
\item Every multi-valued mapping $f$ (satisfying some natural requirements) from a finitely additive probability space $(\Theta,\F,P)$ to $(\Omega,\A)$, with $\F$ an algebra on $\Theta$, induces a belief function $Bel$ on $\A$ \cite{pv-jmaa2018}.
\item Every belief function $Bel$ on $\A$ induces a closed (in the product topology) convex set of finitely additive probability measures on $\A$ referred to as {\it core}:
$$
\core(Bel) = \{P\,:\, \mbox{$P$ is a f.a. probability on $\A$, $Bel \le P$}\}.
$$
It actually holds that $Bel = \min \core(Bel)$, where the minimum is pointwise on the elements of $\A$ \cite{schmeidler2}.
\end{enumerate}

A fifth connection with finitely additive probabilities deserves a particular attention as it allows an integral representation of any belief function. In the case of a finite $\A$, a belief function $Bel$ on $\A$ is in bijection with its {\it M\"obius inverse} \cite{cj-2mon,grabisch},
defined for every $A \in \A$ as
$$
m_{Bel}(A) = \sum_{\substack{B \subseteq A\\ B\in\A}} (-1)^{|\{C_r \in \atoms(\A) \,:\, C_r \subseteq A \setminus B\}|} Bel(B),
$$
also called {\it basic probability assignment}.
The function $m_{Bel}: \A \to [0,1]$ is such that  $m_{Bel}(\emptyset)=0$, $\sum_{A\in\A}m_{Bel}(A)=1$,
and, for every $A \in \A$,
$$
Bel(A) = \sum_{\substack{B\subseteq A\\ B\in\A}}m_{Bel}(B).
$$
The elements of $\A$ where $m_{Bel}$ is positive are said {\it focal elements} \cite{shafer} and a $Bel$ is a probability measure if and only if its focal elements belong to $\atoms(\A)$.

Notice that, disregarding $m_{Bel}(\emptyset) = 0$, $m_{Bel}$ can be formally viewed as a probability distribution on $\U = \A \setminus \{\emptyset\}$ inducing a probability measure $\mu_{Bel}$ on $\wp(\U)$, setting $\mu_{Bel}(\{B\}) = m_{Bel}(B)$ for every $B \in \U$, and extending $\mu_{Bel}$ by additivity. This implies that, in the case of a finite $\A$, there is a bijection between the class of belief functions on $\A$ and the class of probability measures on $\wp(\U)$ setting, for every $A \in \A$
$$
Bel(A) =
\sum_{B \in \U} \left(\min_{\omega \in B}\1_A(\omega)\right) \mu_{Bel}(\{B\}).
$$

Such relation has been generalized to an arbitrary $\Omega$ and an arbitrary algebra $\A$ contained in $\wp(\Omega)$ in \cite{gs-mobius,m-mobius}. Theorem~A in \cite{gs-mobius} shows that every belief function $Bel$ on $\A$ is in bijection with a finitely additive probability measure $\mu_{Bel}$ defined on an algebra $\AA$ possibly strictly contained in $\wp(\U)$, where we still denote $\U = \A \setminus \{\emptyset\}$. The algebra $\AA$ is built as follows. For every $A \in \A$, define $\tilde{A} \in \wp(\U)$ setting
$$
\tilde{A} = \{B \in \U \,:\, B \subseteq A\},
$$
and let $\AA = \algebra(\{\tilde{A} \,:\, A \in \U\})$.

Hence, we get that every belief function $Bel$ on $\A$ is in bijection with a finitely additive probability $\mu_{Bel}$ on $\AA$ such that, for every $A \in \A$
$$
Bel(A) = \int_\U \left(\min_{\omega \in B}\1_A(\omega)\right) \mu_{Bel}(\dd B),
$$
where the integral is of Stieltjes type \cite{bhaskara}.
The finitely additive probability $\mu_{Bel}$ on $\AA$ can still be called the {\it M\"obius inverse} of $Bel$. Notice that, when $\A$ is finite it holds $\AA = \wp(\U)$, thus we recover the classical definition.

\section{Notions of coherence for belief functions}
\label{sec:coherence}
The previous section highlights that conditions {\it (i)} and {\it (ii)} characterizing a belief function essentially rely on the fact that the domain is an algebra. Nevertheless, in decision problems we often need to manage partial numerical evaluations of uncertainty. In such cases conditions {\it (i)} and {\it (ii)} are no more sufficient to guarantee that the assessed values are consistent with a belief function, therefore we need a stronger requirement encoded in a suitable notion of coherence.

In this section, $\Omega$ still denotes an arbitrary non-empty set of states of the world, while $\E \subseteq \wp(\Omega)$ is an arbitrary non-empty family of events. For any non-empty subset $\G \subseteq \E$, we denote $\A_\G = \algebra(\G)$ and $\U_\G = \A_\G \setminus \{\emptyset\}$. Further, we let $\AA_\G = \algebra(\{\tilde{A} \,:\, A \in \U_\G\})$, where $
\tilde{A} = \{B \in \U_\G \,:\, B \subseteq A\}$, for every $A \in \A_\G$.

A {\it belief assessment} is then a function $Bel:\E \to [0,1]$.

In the following we propose different notions to characterize coherence of a belief assessment, seen as a subjective quantification of uncertainty.

As is well-known, the paradigm of coherence is originally due to de~Finetti \cite{definetti} in the context of finitely additive probability measures. Following this idea the paradigm of coherence has been extended also to other uncertainty measures (see, e.g., \cite{cs-libro,walley-lp,walley-libro,williams}). In particular in the context of belief functions we recall \cite{cv2008,Jaffray-Bel}.

As motivated in \cite{Jaffray-Bel}, in situations depending, for instance, on contractual clauses, the occurrence of some events could remain undetermined to the observer, because of a lack of information, at the dates at which the clauses take effect. In other terms, in such situations, that we qualify as {\it partially resolving uncertainty}, we can acquire the information that an event $B \neq \emptyset$ has occurred but we may not be able to identify the true state of the world.
For this reason, in what follows, given an event $E$, to every event $B \neq \emptyset$ we associate the quantity
$$
\min\limits_{\omega \in B} \1_E(\omega) =
\left\{
\begin{array}{ll}
1, & \mbox{if $B \subseteq E$},\\
0, & \mbox{otherwise},
\end{array}
\right.
$$
that acts as a generalized indicator, where uncertainty is resolved pessimistically. Notice that, if we consider an algebra $\A \subseteq \wp(\Omega)$, for every $B \in \A \setminus \{\emptyset\}$, the function $u_{B}:\A \to \{0,1\}$ defined, for every $E \in \A$, as
$$
u_B(E) =
\min\limits_{\omega \in B} \1_E(\omega),
$$
is said {\it vacuous belief function} \cite{deCooman-book} or {\it unanimity game} \cite{grabisch} on $B$.

The first two notions of coherence we introduce are consistency notions that assure the extendibility of a belief assessment to a belief function defined on an algebra. In particular, the second notion of coherence profits from the integral representation of a belief function through a finitely additive probability. Both conditions look for a global extension and do not provide any operational tool for their verification.

\begin{definition}
A belief assessment $Bel:\E \to [0,1]$ is {\bf coherent-1} if there exists a belief function $Bel':\A_\E \to [0,1]$ such that, for every $E \in \E$, it holds
$$
Bel(E) = Bel'(E).
$$
\end{definition}

\begin{definition}
A belief assessment $Bel:\E \to [0,1]$ is {\bf coherent-2} if there exists a finitely additive probability $\mu:\AA_\E \to [0,1]$ such that, for every $E \in \E$, it holds
$$
Bel(E) =  \int_{\U_\E} \left(\min_{\omega \in B}\1_E(\omega)\right) \mu(\dd B).
$$
\end{definition}

Results on the extendibility of a belief function to a larger domain can be found in \cite{walley-lp}, where $\E$ is assumed to be an algebra, and \cite{dCtm}, where $\E$ is assumed to be a lattice with respect to intersection and union. Our approach remains, nevertheless, more general.

The following condition is still a consistency notion that, contrary to the previous two conditions, has an operational appeal as it requires to solve a linear system for every finite subfamily of events.

\begin{definition}
A belief assessment $Bel:\E \to [0,1]$ is {\bf coherent-3} if, for every $n \in \NN$ and every $\F = \{E_1,\ldots,E_n\} \subseteq \E$, the following linear system with unknowns $x_B$, for every $B \in \U_\F$, is compatible
$$
\S_\F:
\left\{
\begin{array}{ll}
\displaystyle{\sum_{B \in \U_\F} \left(\min_{\omega \in B}\1_{E_i}(\omega)\right) x_B = Bel(E_i)}, & \mbox{for $i=1,\ldots,n$},\\[1ex]
\displaystyle{\sum_{B \in \U_\F} x_B = 1},\\
x_B \ge 0, &\mbox{for every $B \in \U_\F$}.
\end{array}
\right.
$$
\end{definition}

The fourth notion we introduce has a betting scheme interpretation, analogous to that proposed by de~Finetti for probabilities \cite{definetti}, but working under partially resolving uncertainty. Given a finite subfamily $\F = \{E_1,\ldots,E_n\} \subseteq \E$, a {\it bet} on $E_i$ with stake $\lambda_i \in \RR$ produces a gain defined, for every $B \in \U_\F$, as
$$
\lambda_i\left(\left(\min_{\omega \in B} \1_{E_i}(\omega)\right) - Bel(E_i)\right),
$$
where $Bel(E_i)$ can be interpreted as the amount paid to participate to the bet. The bets on events in $\F$ can be combined giving rise to the gain defined, for every $B \in \U_\F$, as
$$
G_\F(B) = \sum_{i = 1}^n \lambda_i\left(\left(\min_{\omega \in B} \1_{E_i}(\omega)\right) - Bel(E_i)\right).
$$
Notice that the gain is defined on $\U_\F$, whose elements express the possible partial information we may acquire when uncertainty is resolved. The combination of bets is coherent if the gain $G_\F$ is not uniformly negative over $\U_\F$, otherwise we will have a so-called {\it Dutch book}. Of course, the notion of Dutch book appearing here is a generalization of that introduced for probabilities by de~Finetti \cite{definetti}, where the gain $G_\F$ is restricted only to the elements of $\atoms(\algebra(\F))$ which are a proper subset of $\U_\F$ expressing completely resolving uncertainty.

\begin{definition}
A belief assessment $Bel:\E \to [0,1]$ is {\bf coherent-4} if, for every $n \in \NN$ and every $\F = \{E_1,\ldots,E_n\} \subseteq \E$, for every $\lambda_1,\ldots,\lambda_n \in \RR$, the function $G_\F:\U_\F \to \RR$ defined, for every $B \in \U_\F$, as
$$
G_\F(B) = \sum_{i = 1}^n \lambda_i\left(\left(\min_{\omega \in B} \1_{E_i}(\omega)\right) - Bel(E_i)\right),
$$
is such that $
\displaystyle{\max_{B \in \U_\F} G_\F(B)\ge 0}$.
\end{definition}

We point out the resemblance of condition coherence-4 with the {avoiding sure loss condition} for a lower probability given by Walley \cite{walley-lp,walley-libro} (see also \cite{deCooman-book}). It is important to notice that Walley's condition deals with a gain function defined on $\Omega$ while in our case, working under partially resolving uncertainty, the gain function is defined on $\U_\F$.


In analogy with the classical betting scheme due to de~Finetti, condition coherence-4 can be subject to a strategic aspect, i.e., an agent could try to guess the choice of the ``opponent'' and modify his belief assessment accordingly \cite{definetti,ssk-scoring}.

The last notion of coherence we propose is the generalization 
of the penalty criterion proposed in \cite{definetti}. Given a finite subfamily $\F = \{E_1,\ldots,E_n\} \subseteq \E$, the assessed $Bel(E_i)$ on $E_i$ causes to the agent a {\it penalty} defined, for every $B \in \U_\F$, as
$$
\left(\left(\min_{\omega \in B} \1_{E_i}(\omega)\right) - Bel(E_i)\right)^2.
$$
Hence, the restriction of the assessment $Bel$ to $\F$ will result in a global penalty defined, for every $B \in \U_\F$, as
$$
L_\F(B) = \sum_{i = 1}^n \left(\left(\min_{\omega \in B} \1_{E_i}(\omega)\right) - Bel(E_i)\right)^2.
$$
Notice that the global penalty is defined on $\U_\F$, whose elements express the possible partial information we may acquire when uncertainty is resolved. The assessment is coherent if there is no distinct assessment $Bel^*:\F \to [0,1]$ such that the corresponding global penalty $L_\F^*$ is uniformly lower than $L_\F$ over $\U_\F$.

\begin{definition}
A belief assessment $Bel:\E \to [0,1]$ is {\bf coherent-5} if, for every $n \in \NN$ and every $\F = \{E_1,\ldots,E_n\} \subseteq \E$, there is no distinct assessment $Bel^*:\F \to [0,1]$ such that the functions $L_\F^*,L_\F:\U_\F \to \RR$ defined, for every $B \in \U_\F$, as
\begin{eqnarray*}
L_\F^*(B) &=& \sum_{i = 1}^n \left(\left(\min_{\omega \in B} \1_{E_i}(\omega)\right) - Bel^*(E_i)\right)^2,\\
L_\F(B) &=& \sum_{i = 1}^n \left(\left(\min_{\omega \in B} \1_{E_i}(\omega)\right) - Bel(E_i)\right)^2,
\end{eqnarray*}
satisfy, for every $B \in \U_\F$, $L_\F^*(B) < L_\F(B)$.
\end{definition}

In analogy to what has been proven in the case of probability assessments (see \cite{gilio} and \cite{scoring-rules}) condition coherence-5 can be equivalently formulated by requiring the non-existence of a distinct belief assessment $Bel^* : \F \to [0,1]$ such that, for every $B \in \U_\F$, $L_\F^*(B) \le L_\F(B)$.

In the particular case $\E$ is a finite set, conditions coherence-3, coherence-4 and coherence-5 introduced in the previous definitions can be simplified by requiring them to hold only for $\F = \E$. Indeed, if they hold for $\F = \E$, then it is easily proved that they hold also for every non-empty subset $\F \subseteq \E$.

Limiting to a finite $\Omega$ and $\E = \wp(\Omega)$, condition coherence-4 has been introduced in \cite{Jaffray-Bel}, while the case of a finite $\Omega$ and an arbitrary $\E$ has been considered in \cite{Regoli1994}.
Again, limiting to a finite $\Omega$ and $\E = \wp(\Omega)$, it is easy to show that our condition coherence-4 can be expressed in terms of the notion of $B$-consistency for a coherent betting function $R:\RR^\Omega \to \{0,1\}$ introduced in \cite{KerkMeest}.
Nevertheless, our condition coherence-4 is stronger than all conditions proposed in \cite{Jaffray-Bel,KerkMeest,Regoli1994} since no assumption is made neither on $\Omega$ nor on $\E$.
We point out that a dual version of conditions coherence-1 and coherence-3 can be traced back to \cite{ccv-kyb2014}, where a more general condition dealing with a conditional plausibility assessment is considered. This is due to duality between belief and plausibility functions (see, e.g., \cite{grabisch,shafer}) and the identification of $E_i$ with $E_i|\Omega$. Morover, the dual of condition coherence-4 turns out to be a particular case of a more general condition of coherence we introduced in \cite{bumi} for a conditional completely alternating Choquet expectation assessment. 
As far as we know, conditions coherence-2 and coherence-5 are new in the literature.

The following theorem shows that all notions of coherence introduced so far are equivalent.

\begin{theorem}
\label{th:cohe-bel}
For a belief assessment $Bel:\E \to [0,1]$ the following statements are equivalent:
\begin{description}
\item[\it (i)] $Bel$ is coherent-1;
\item[\it (ii)] $Bel$ is coherent-2;
\item[\it (iii)] $Bel$ is coherent-3;
\item[\it (iv)] $Bel$ is coherent-4;
\item[\it (v)] $Bel$ is coherent-5.
\end{description}
\end{theorem}
\begin{proof}
{\it (i)} $\Longleftrightarrow$ {\it (ii)} The equivalence follows by Theorem~A in \cite{gs-mobius}.

{\it (i)} $\Longrightarrow$ {\it (iii)} If $Bel':\A_\E \to [0,1]$ is a belief function extending $Bel$, then, for every $n \in \NN$ and every $\F = \{E_1,\ldots,E_n\} \subseteq \E$, its restriction to $\A_\F$ is a belief function $Bel''$ having M\"obius inverse $m_{Bel''}$. Setting $x_B = m_{Bel''}(B)$, for every $B \in \U_\F$, we get a solution of system $\S_\F$.

{\it (iii)} $\Longrightarrow$ {\it (i)} Suppose, for every $n \in \NN$ and every $\F = \{E_1,\ldots,E_n\} \subseteq \E$, the system $\S_\F$ is compatible. Then, if ${\bf x}$ is a solution with component $x_B$, for every $B \in \U_\F$, setting $m_{Bel''}(\emptyset) = 0$ and $m_{Bel''}(B) = x_B$, for every $B \in \U_\F$, we get the M\"obius inverse of a belief function $Bel''$ extending $Bel_{|\F}$ on the whole $\A_\F$. Notice that every belief function $Bel''$ on $\A_\F$ extending $Bel_{|\F}$ is obtained through a solution of system $\S_\F$. Denote by $\BB_\F$ the set of mappings from $\A_\E$ to $[0,1]$ whose restriction to $\A_\F$ is a belief function extending $Bel_{|\F}$, where the restriction is determined by a solution of $\S_\F$. The set $\BB_\F$ is a non-empty \red{closed} subset of the \red{compact} set $[0,1]^{\A_\E}$ endowed with the product topology. It is easily seen that the family
$$
\{\BB_\F \,:\, \F = \{E_1,\ldots,E_n\} \subseteq \E,n \in \NN\},
$$
possesses the finite intersection property, thus it holds
$$
\bigcap\{\BB_\F \,:\, \F = \{E_1,\ldots,E_n\} \subseteq \E,n \in \NN\} \neq \emptyset
$$
and so there exists $Bel' \in \bigcap\{\BB_\F \,:\, \F = \{E_1,\ldots,E_n\} \subseteq \E,n \in \NN\}$ which is a belief function on $\A_\E$ extending $Bel$.

{\it (iii)} $\Longleftrightarrow$ {\it (iv)} The proof can be traced back to \cite{Regoli1994} in case of a finite $\E$: we provide the proof for an arbitrary $\E$ for the sake of completeness. For $n \in \NN$ and $\F = \{E_1,\ldots,E_n\} \subseteq \E$, fix an enumeration of $\U_\F = \{A_1,\ldots,A_{2^m - 1}\}$, where $m$ is the cardinality of $\atoms(\A_\F)$. Then system $\S_\F$ can be written in matrix form as
$$
\S_\F:
\left\{
\begin{array}{ll}
A{\bf x} = {\bf b},\\
{\bf x} \ge {\bf 0},
\end{array}
\right.
$$
where ${\bf x} = (x_1,\ldots,x_{2^m-1})^T \in \RR^{(2^m - 1) \times 1}$ is an unknown column vector, $A = (a_{ij}) \in \RR^{(n+1) \times (2^m - 1)}$ is the coefficient matrix with
$$
\begin{array}{ll}
a_{ij} = \displaystyle{\min_{\omega \in A_j} \1_{E_i}(\omega)}, & \mbox{ for $i=1,\ldots,n$, $j = 1,\ldots,2^m-1$},\\
a_{(n+1)j} = 1, &\mbox{ for $j = 1,\ldots,2^m-1$},
\end{array}
$$
and ${\bf b} = (Bel(E_1),\ldots,Bel(E_n),1)^T \in \RR^{(n+1) \times 1}$.

By Farkas' lemma \cite{mangasarian}, system $\S_\F$ is compatible if and only if the following system $\S_\F^*$ is not compatible
$$
\S_\F^*:
\left\{
\begin{array}{ll}
{\bf y} A \le {\bf 0},\\
{\bf y}{\bf b} > 0,
\end{array}
\right.
$$
where ${\bf y} = (\lambda_1,\ldots,\lambda_n,y_{n+1}) \in \RR^{1 \times (n+1)}$ is an unknown row vector. It holds that ${\bf y} A \in \RR^{1 \times (2^m - 1)}$ and, for $j=1,\ldots,(2^m - 1)$, the $j$th column of constraint ${\bf y} A \le {\bf 0}$ is
$$
\sum_{i=1}^n \lambda_i\left(\min_{\omega \in A_j}\1_{E_i}(\omega)\right) + y_{n+1} \le 0,
$$
moreover, subtracting the positive quantity ${\bf y}{\bf b}$ we get
$$
\sum_{i=1}^n \lambda_i\left(\left(\min_{\omega \in A_j}\1_{E_i}(\omega)\right) - Bel(E_i)\right) < 0.
$$
Thus, condition {\it (iii)} is equivalent to the existence of a $j \in \{1,\ldots,2^m - 1\}$ such that the above inequality does not hold, which, in turn, is equivalent to {\it (iv)}.

{\it (iii)} $\Longleftrightarrow$ {\it (v)}
For $n \in \NN$ and $\F = \{E_1,\ldots,E_n\} \subseteq \E$, fix an enumeration of $\U_\F = \{A_1,\ldots,A_{2^m - 1}\}$, where $m$ is the cardinality of $\atoms(\A_\F)$, so system $\S_\F$ can be written in matrix form as in the previous point.

For $j = 1, \ldots, 2^m - 1$, define the column vector ${\bf e}^j = (e^j_1,\ldots,e^j_n)^T \in \RR^{n \times 1}$ setting
$$
e^j_i = \min_{\omega \in A_j} \1_{E_i}(\omega) =
\left\{
\begin{array}{ll}
1, & \mbox{if $A_j \subseteq E_i$,}\\
0, & \mbox{otherwise,}\\
\end{array}
\right.
$$
moreover, let ${\bf d} = (Bel(E_1),\ldots,Bel(E_n))^T \in \RR^{n \times 1}$. It holds that ${\bf x} \in \RR^{(2^m - 1) \times 1}$ is a solution of system $\S_\F$ if and only if ${\bf x} \ge {\bf 0}$, $\sum_{j=1}^{2^m-1} x_j = 1$ and
$$
{\bf d} = \sum_{j=1}^{2^m-1} x_j {\bf e}^j,
$$
that is ${\bf d}$ is the convex combination of vectors ${\bf e}^j$'s with weights given by ${\bf x}$. In other terms, condition {\it (iii)} is equivalent to the fact that the vector ${\bf d}$ belongs to the convex hull of vectors ${\bf e}^j$'s, that is ${\bf d} \in \conv(\{{\bf e}^1,\ldots,{\bf e}^{2^m-1}\})$.

For vectors ${\bf u} = (u_1,\ldots,u_n)^T,{\bf v}=(v_1,\ldots,v_n)^T \in \RR^{n \times 1}$, denote by
$$
d_E({\bf u},{\bf v}) = \sqrt{\sum_{i=1}^n (u_i - v_i)^2},
$$
their Euclidean distance.

If ${\bf x} \in \RR^{(2^m - 1) \times 1}$ has non-negative components summing up to 1, then let the moment of inertia $I_{\bf d}$ with respect to ${\bf d}$ of the vector of weights ${\bf x}$ be
$$
I_{\bf d} = \sum_{j=1}^{2^m-1}x_j d_E({{\bf e}^j},{\bf d})^2.
$$
Now, for a distinct $Bel^*:\F \to [0,1]$, let ${\bf d}^* = (Bel^*(E_1),\ldots,Bel^*(E_n))^T \in \RR^{n \times 1}$ and define analogously
$$
I_{{\bf d}^*} = \sum_{j=1}^{2^m-1}x_j d_E({{\bf e}^j},{\bf d}^*)^2.
$$
Notice that, for $j=1,\ldots,2^m-1$, $L_\F^*(A_j) = d_E({{\bf e}^j},{\bf d}^*)^2$ and $L_\F(A_j) = d_E({{\bf e}^j},{\bf d})^2$.

We first show that condition {\it (iii)} implies {\it (v)}.
If ${\bf x} \in \RR^{(2^m - 1) \times 1}$ is a solution of system $\S_\F$, then
$$
I_{{\bf d}^*} - I_{{\bf d}} = \sum_{j=1}^{2^m-1}x_j\left( d_E({{\bf e}^j},{\bf d}^*)^2 - d_E({{\bf e}^j},{\bf d})^2\right) = d_E({\bf d}^*,{\bf d})^2 > 0,
$$
so, for at least a $j \in \{1,\ldots,2^m-1\}$ it holds
$d_E({{\bf e}^j},{\bf d}^*)^2 > d_E({{\bf e}^j},{\bf d})^2$.
Since $Bel^*$ is arbitrary, if condition {\it (iii)} holds, then we cannot find an assessment $Bel^*$ such that $L_\F^*(A_j) < L_\F(A_j)$, for $j=1,\ldots,2^m-1$, so condition {\it (v)} holds.

Finally, we show that condition {\it (v)} does not hold if condition {\it (iii)} does not hold. Hence, assume there are $n \in \NN$ and $\F = \{E_1,\ldots,E_n\} \subseteq \E$ such that system $\S_\F$ is not solvable, that is ${\bf d} \notin \conv(\{{\bf e}^1,\ldots,{\bf e}^{2^m-1}\})$.

As follows from results in \cite{gs-sanfilippo},
the squared Euclidean distance coincides with the Bregman divergence determined by the Brier quadratic scoring rule, which is a bounded (strictly) proper scoring rule (see \cite{scoring-rules}). Hence, by Proposition~3 in \cite{scoring-rules}, since ${\bf d} \notin \conv(\{{\bf e}^1,\ldots,{\bf e}^{2^m-1}\})$, there exists a unique element of $\conv(\{{\bf e}^1,\ldots,{\bf e}^{2^m-1}\})$ minimizing the squared Euclidean distance with respect to ${\bf d}$, said projection of ${\bf d}$ onto $\conv(\{{\bf e}^1,\ldots,{\bf e}^{2^m-1}\})$.
Let ${\bf d}^* = (Bel^*(E_1),\ldots,Bel^*(E_n))^T$ be the projection of ${\bf d}$ onto $\conv(\{{\bf e}^1,\ldots,{\bf e}^{2^m-1}\})$, that is
$${\bf d}^* = \displaystyle{\arg\min_{{\bf u} \in \conv(\{{\bf e}^1,\ldots,{\bf e}^{2^m-1}\})} d_E({\bf u},{\bf d})^2}.
$$
By Proposition~3 in \cite{scoring-rules} it holds that, for $j=1,\ldots,2^m-1$, we have
$$
d_E({\bf e}^j,{\bf d}^*)^2 + d_E({\bf d}^*,{\bf d})^2 \le d_E({\bf e}^j,{\bf d})^2,
$$
moreover, since $d_E({\bf d}^*,{\bf d})^2 > 0$, it holds
$$
d_E({\bf e}^j,{\bf d}^*)^2 < d_E({\bf e}^j,{\bf d})^2.
$$
This implies that with such a $Bel^*$, $L_\F^*(A_j) < L_\F(A_j)$, for $j=1,\ldots,2^m-1$, so condition {\it (v)} does not hold.
\hfill$\square$
\end{proof}

By virtue of Theorem~\ref{th:cohe-bel}, we say that a belief assessment $Bel:\E \to [0,1]$ is {\it coherent} if one (and hence all) of the previous notions of coherence holds, otherwise it is said {\it incoherent}.

Referring to results appearing in \cite{walley-lp} and \cite{dCtm}, a coherent belief assessment $Bel$ on $\E$ can be extended (generally not in a unique way) to a belief function on $\wp(\Omega)$. In turn, such a belief function gives rise to a completely monotone lower expectation operator on the set of bounded real functions on $\Omega$, defined through the Choquet integral.

\section{Special cases}
\label{sec:specialcases}
\subsection{Finitely additive probability measures}
As already recalled in Section~\ref{sec:preliminaries}, finitely additive probability measures on an algebra $\A$ form a distinguished subclass of the class of belief functions on $\A$.

In general, if $\E$ is an arbitrary non-empty family of events and $P:\E\to[0,1]$ is a probability assessment, the notions of coherence introduced in Section~\ref{sec:coherence} can be specialized to work inside the subclass of finitely additive probabilities. The obtained conditions of coherence exactly coincide with those proposed by de~Finetti \cite{definetti}. In what follows, if $\F$ is a finite non-empty set of events, then $\C_\F$ denotes the set of {\it atoms} generated by them, i.e., $\C_\F = \atoms(\algebra(\F))$.

As we pointed out in the previous section, finitely additive probability theory can be considered as a special case of Dempster-Shafer theory. 
In practice, this has a direct impact on conditions coherence-3, coherence-4 and coherence-5 where $\U_\F$, that collects all the possible partial information we may acquire on the occurrence of an event generated from $\F$, is replaced by the set of atoms $\C_\F$. 

In detail, a probability assessment $P:\E \to [0,1]$ is:
\begin{description}
\item[\bf Coherent-1P:] if there exists a finitely additive probability $P':\A_\E \to [0,1]$ such that, for every $E \in \E$, it holds
$$
P(E) = P'(E).
$$
\item[\bf Coherent-2P:] if there exists a finitely additive probability $\mu:\AA_\E \to [0,1]$ such that the mapping $A \mapsto \mu(\tilde{A})$, for every $A \in \A_\E$, is a finitely additive probability and, for every $E \in \E$,
$$
P(E) =  \int_{\U_\E} \left(\min_{\omega \in B}\1_E(\omega)\right) \mu(\dd B).
$$
\item[\bf Coherent-3P:] if, for every $n \in \NN$ and every $\F = \{E_1,\ldots,E_n\} \subseteq \E$, the following linear system with unknowns $x_C$, for every $C \in \C_\F$, is compatible
$$
\S_\F:
\left\{
\begin{array}{ll}
\displaystyle{\sum_{C \in \C_\F} \left(\min_{\omega \in C}\1_{E_i}(\omega)\right) x_C = P(E_i)}, & \mbox{for $i=1,\ldots,n$},\\[1ex]
\displaystyle{\sum_{C \in \C_\F} x_C = 1},\\
x_C \ge 0, &\mbox{for every $C \in \C_\F$}.
\end{array}
\right.
$$
\item[\bf Coherent-4P:] if, for every $n \in \NN$ and every $\F = \{E_1,\ldots,E_n\} \subseteq \E$, for every $\lambda_1,\ldots,\lambda_n \in \RR$, the function $G_\F:\C_\F \to \RR$ defined, for every $C \in \C_\F$, as
$$
G_\F(C) = \sum_{i = 1}^n \lambda_i\left(\left(\min_{\omega \in C} \1_{E_i}(\omega)\right) - P(E_i)\right),
$$
is such that $
\displaystyle{\max_{C \in \C_\F} G_\F(C)\ge 0}$.
\item[\bf Coherent-5P:] if, for every $n \in \NN$ and every $\F = \{E_1,\ldots,E_n\} \subseteq \E$, there is no distinct assessment $P^*:\F \to [0,1]$ such that the functions $L_\F^*,L_\F:\C_\F \to \RR$ defined, for every $C \in \C_\F$, as
\begin{eqnarray*}
L_\F^*(C) &=& \sum_{i = 1}^n \left(\left(\min_{\omega \in C} \1_{E_i}(\omega)\right) - P^*(E_i)\right)^2,\\
L_\F(C) &=& \sum_{i = 1}^n \left(\left(\min_{\omega \in C} \1_{E_i}(\omega)\right) - P(E_i)\right)^2,
\end{eqnarray*}
satisfy, for every $C \in \C_\F$, $L_\F^*(C) < L_\F(C)$.
\end{description}

Also in the case of a probability assessment we have that the notions of coherence given above are equivalent.

\begin{corollary}
\label{cor:equiv-prob}
For a probability assessment $P:\E \to [0,1]$ the following statements are equivalent:
\begin{description}
\item[\it (i)] $P$ is coherent-1P;
\item[\it (ii)] $P$ is coherent-2P;
\item[\it (iii)] $P$ is coherent-3P;
\item[\it (iv)] $P$ is coherent-4P;
\item[\it (v)] $P$ is coherent-5P.
\end{description}
\end{corollary}
\begin{proof}
The equivalence between {\it (i)}, {\it (iii)}, {\it (iv)}, and {\it (v)} is due to de~Finetti \cite{definetti}. The equivalence between {\it (i)} and {\it (ii)} is an immediate consequence of Theorem~A and Lemma~4.1.2 in \cite{gs-mobius}.
\hfill$\square$
\end{proof}

The above notions of coherence for a probability assessment imply the corresponding notions for a belief assessment. In particular, in view of Corollary~\ref{cor:equiv-prob}, we say that a probability assessment $P:\E \to [0,1]$ is {\it coherent-P} if one (and hence all) of the previous notions of coherence holds, otherwise it is said {\it incoherent-P}.

\subsection{Finitely minitive necessity measures}
If $\A$ is an algebra of substes of $\Omega$, a {\it finitely minitive necessity measure} \cite{dubprad} is a mapping $N:\A \to [0,1]$ such that:
\begin{description}
\item[\it (i)] $N(\emptyset) = 0$ and $N(\Omega) = 1$;
\item[\it (ii)] $N(A \cap B) = \min\{N(A),N(B)\}$, for every $A,B \in \A$.
\end{description}
It turns out that every finitely minitive necessity measure $N$ is a belief function, indeed (see, e.g., \cite{shafer} and \cite{grabisch}), for every $k\geq 2$ and every $A_1,\ldots,A_k \in \A$, the restriction of $N$ to $\algebra(\{A_1,\ldots,A_k\})$ is such that
$$
N \left(\bigcup_{i=1}^k A_i\right) \geq \sum_{\emptyset \neq I \subseteq \{1,\ldots,k\}} (-1)^{|I| + 1} N\left(\bigcap_{i \in I} A_i\right).
$$

In the original work by Shafer \cite{shafer}, limiting to a finite $\Omega$ and $\A = \wp(\Omega)$, a (finitely minitive) necessity measure is called {\it consonant belief function}. In particular, in such a finite setting, this translates to a M\"obius inverse with nested focal elements (see, e.g., \cite{dubprad-when,grabisch,shafer}).

Given a finite non-empty set of events $\F$,
if the set of atoms generated by $\F$ is
$\C_\F = \atoms(\algebra(\F)) = \{C_1,\ldots,C_m\}$, denote by $\chains(\U_\F)$ the collection of subfamilies of $\U_\F$ such that $\D_\F = \{D_1,\ldots,D_m\} \in \chains(\U_\F)$ if and only if $D_1 = C_{i_1}$, $D_2 = C_{i_1} \cup C_{i_2}, \ldots, D_m = C_{i_1} \cup \cdots \cup C_{i_m} = \Omega$. Notice that the elements of $\chains(\U_\F)$ are in one-to-one correspondence with permutations of $\C_\F$.

Finitely minitive necessity theory can be considered as a special case of Dempster-Shafer theory. 
In practice, this has a direct impact on conditions coherence-3, coherence-4 and coherence-5 where $\U_\F$, that collects all the possible partial information we may acquire on the occurrence of an event generated from $\F$, is replaced by a family $\D_\F \in \chains(\U_\F)$. In other terms, here we assume partially resolving uncertainty but we require consonance, i.e., we focus on nested sets of events generated from the events in $\F$.

In detail, a necessity assessment $N:\E \to [0,1]$ is:
\begin{description}
\item[\bf Coherent-1N:] if there exists a finitely minitive necessity measure $N':\A_\E \to [0,1]$ such that, for every $E \in \E$, it holds
$$
N(E) = N'(E).
$$
\item[\bf Coherent-2N:] if there exists a finitely additive probability $\mu:\AA_\E \to [0,1]$ such that the mapping $A \mapsto \mu(\tilde{A})$, for every $A \in \A_\E$, is a finitely minitive necessity measure and, for every $E \in \E$,
$$
N(E) =  \int_{\U_\E} \left(\min_{\omega \in B}\1_E(\omega)\right) \mu(\dd B).
$$
\item[\bf Coherent-3N:] if, for every $n \in \NN$ and every $\F = \{E_1,\ldots,E_n\} \subseteq \E$, there exists $\D_\F \in \chains(\U_\F)$ such that the following linear system with unknowns $x_D$, for every $D \in \D_\F$, is compatible
$$
\S_\F:
\left\{
\begin{array}{ll}
\displaystyle{\sum_{D \in \D_\F} \left(\min_{\omega \in D}\1_{E_i}(\omega)\right) x_D = N(E_i)}, & \mbox{for $i=1,\ldots,n$},\\[1ex]
\displaystyle{\sum_{D \in \D_\F} x_D = 1},\\
x_D \ge 0, &\mbox{for every $D \in \D_\F$}.
\end{array}
\right.
$$
\item[\bf Coherent-4N:] if, for every $n \in \NN$ and every $\F = \{E_1,\ldots,E_n\} \subseteq \E$, there exists $\D_\F \in \chains(\U_\F)$ such that, for every $\lambda_1,\ldots,\lambda_n \in \RR$, the function $G_\F:\D_\F \to \RR$ defined, for every $D \in \D_\F$, as
$$
G_\F(D) = \sum_{i = 1}^n \lambda_i\left(\left(\min_{\omega \in D} \1_{E_i}(\omega)\right) - N(E_i)\right),
$$
is such that $
\displaystyle{\max_{D \in \D_\F} G_\F(D)\ge 0}$.
\item[\bf Coherent-5N:] if, for every $n \in \NN$ and every $\F = \{E_1,\ldots,E_n\} \subseteq \E$, there exists $\D_\F \in \chains(\U_\F)$ such that there is no distinct assessment $N^*:\F \to [0,1]$ such that the functions $L_\F^*,L_\F:\D_\F \to \RR$ defined, for every $D \in \D_\F$, as
\begin{eqnarray*}
L_\F^*(D) &=& \sum_{i = 1}^n \left(\left(\min_{\omega \in D} \1_{E_i}(\omega)\right) - N^*(E_i)\right)^2,\\
L_\F(D) &=& \sum_{i = 1}^n \left(\left(\min_{\omega \in D} \1_{E_i}(\omega)\right) - N(E_i)\right)^2,
\end{eqnarray*}
satisfy, for every $D \in \D_\F$, $L_\F^*(D) < L_\F(D)$.
\end{description}

We notice that a conditional version of condition coherence-1N has been introduced in \cite{cv-ijar2007}. As far we know, all the other conditions are new in the literature.

Also in the case of a necessity assessment we have that the notions of coherence given above are equivalent.

\begin{corollary}
\label{cor:equiv-nec}
For a necessity assessment $N:\E \to [0,1]$ the following statements are equivalent:
\begin{description}
\item[\it (i)] $N$ is coherent-1N;
\item[\it (ii)] $N$ is coherent-2N;
\item[\it (iii)] $N$ is coherent-3N;
\item[\it (iv)] $N$ is coherent-4N;
\item[\it (v)] $N$ is coherent-5N.
\end{description}
\end{corollary}
\begin{proof}
{\it (i)} $\Longleftrightarrow$ {\it (ii)} The equivalence follows by Theorem~A and Lemma~4.1.2 in \cite{gs-mobius}.

{\it (i)} $\Longrightarrow$ {\it (iii)} If $N':\A_\E \to [0,1]$ is a finitely minitive necessity measure extending $N$, then, for every $n \in \NN$ and every $\F = \{E_1,\ldots,E_n\} \subseteq \E$, its restriction to $\A_\F$ is a necessity measure $N''$ having M\"obius inverse $m_{N''}$. Thus, by Theorem~7.38 in \cite{grabisch} there exists $\D_\F \in \chains(\U_\F)$ such that setting $x_D = m_{N''}(D)$, for every $D \in \D_\F$, we get a solution of the corresponding system $\S_\F$.

{\it (iii)} $\Longrightarrow$ {\it (i)} Suppose, for every $n \in \NN$ and every $\F = \{E_1,\ldots,E_n\} \subseteq \E$, there exists $\D_\F \in \chains(\U_\F)$ such that the system $\S_\F$ is compatible. Then, if ${\bf x}$ is a solution with component $x_D$, for every $D \in \D_\F$, setting $m_{N''}(D) = x_D$, for every $D \in \D_\F$, and $m_{N''}(A) = 0$, for every $A \in \A_\F \setminus \D_\F$, we get the M\"obius inverse of a necessity measure $N''$ extending $N_{|\F}$ on the whole $\A_\F$. Notice that, by Theorem~7.38 in \cite{grabisch}, every necessity measure $N''$ on $\A_\F$ extending $N_{|\F}$ is obtained for a $\D_\F \in \chains(\U_\F)$ through a solution of the corresponding system $\S_\F$. 
Denote by $\NNN_\F$ the set of mappings from $\A_\E$ to $[0,1]$ whose restriction to $\A_\F$ is a necessity measure extending $N_{|\F}$, where the restriction is determined by a $\D_\F \in \chains(\U_\F)$ through a solution of the corresponding $\S_\F$. The set $\NNN_\F$ is a non-empty \red{closed} subset of the \red{compact} set $[0,1]^{\A_\E}$ endowed with the product topology. It is easily seen that the family
$$
\{\NNN_\F \,:\, \F = \{E_1,\ldots,E_n\} \subseteq \E,n \in \NN\},
$$
possesses the finite intersection property, thus it holds
$$
\bigcap\{\NNN_\F \,:\, \F = \{E_1,\ldots,E_n\} \subseteq \E,n \in \NN\} \neq \emptyset
$$
and so there exists $N' \in \bigcap\{\NNN_\F \,:\, \F = \{E_1,\ldots,E_n\} \subseteq \E,n \in \NN\}$ which is a finitely minitive necessity measure on $\A_\E$ extending $N$.

{\it (iii)} $\Longleftrightarrow$ {\it (iv)} For $n \in \NN$ and $\F = \{E_1,\ldots,E_n\} \subseteq \E$, suppose there exists $\D_\F \in \chains(\U_\F)$ such that the corresponding system $\S_\F$ is compatible. Then, following the same steps of the proof of Theorem~\ref{th:cohe-bel} and working with $\D_\F$ in place of $\U_\F$, we have that solvability of $\S_\F$ is equivalent to $\max\limits_{D \in \D_F}G_\F(D) \ge 0$.

{\it (iii)} $\Longleftrightarrow$ {\it (v)} For $n \in \NN$ and $\F = \{E_1,\ldots,E_n\} \subseteq \E$, suppose there exists $\D_\F \in \chains(\U_\F)$ such that the corresponding system $\S_\F$ is compatible. Then, following the same steps of the proof of Theorem~\ref{th:cohe-bel} and working with $\D_\F$ in place of $\U_\F$, we have that solvability of $\S_\F$ is equivalent to the non-existence of a distinct assessment $N^*:\F \to [0,1]$ such that $L^*_\F(D) < L_\F(D)$, for every $D \in \D_\F$.
\hfill$\square$
\end{proof}

The above notions of coherence for a necessity assessment imply the corresponding notions for a belief assessment. In particular, in view of Corollary~\ref{cor:equiv-nec}, we say that a necessity assessment $N:\E \to [0,1]$ is {\it coherent-N} if one (and hence all) of the previous notions of coherence holds, otherwise it is said {\it incoherent-N}.

\section{Proper scoring rules and  correction of an incoherent assessment}
\label{sec:properscoring}
The notions of coherence-5 (respectively, coherence-5P and coherence-5N) can be generalized recurring to a general proper scoring rule \cite{scoring-rules,ssk-scoring,ssk-imprecise}.

A function $s:\{0,1\} \times [0,1] \to [0,+\infty]$ is said a {\it (strictly) proper scoring rule} if it satisfies:
\begin{description}
\item[\it (i)] $ps(1,x)+(1-p)s(0,x)$ is uniquely minimized at $x = p$, for every $p \in [0,1]$;
\item[\it (ii)] $s(i,x)$ is continuous in the second variable, for every $i \in \{0,1\}$.
\end{description}

A proper scoring rule $s$ is {\it bounded} if $s(i,x)$ is bounded in the second variable, for every $i \in \{0,1\}$.

Two popular choices are the Brier quadratic scoring rule $s_B(i,x) = (i - x)^2$ \cite{Brier}, which is bounded, and the logarithmic scoring rule $s_L(i,x) = -\ln|1-i-x|$ \cite{Good}, which is unbounded.

\begin{definition}
Let $s$ be a fixed proper scoring rule.
A belief assessment $Bel:\E \to [0,1]$ is {\bf coherent-5($s$)} if, for every $n \in \NN$ and every $\F = \{E_1,\ldots,E_n\} \subseteq \E$, there is no distinct assessment $Bel^*:\F \to [0,1]$ such that the functions $L_\F^*,L_\F:\U_\F \to \RR$ defined, for every $B \in \U_\F$, as
\begin{eqnarray*}
L_\F^*(B) &=& \sum_{i = 1}^n s\left(\left(\min_{\omega \in B} \1_{E_i}(\omega)\right),Bel^*(E_i)\right),\\
L_\F(B) &=& \sum_{i = 1}^n s\left(\left(\min_{\omega \in B} \1_{E_i}(\omega)\right),Bel(E_i)\right),
\end{eqnarray*}
satisfy, for every $B \in \U_\F$, $L_\F^*(B) < L_\F(B)$.
\end{definition}
Notice that condition coherence-5 coincides with coherence-5($s_B$) obtained using the Brier quadratic scoring rule $s_B$.

Every bounded proper scoring rule $s$ can be used to define a {\it Bregman divergence} between vectors ${\bf u} = (u_1,\ldots,u_n)^T, {\bf v} = (v_1,\ldots,v_n)^T \in [0,1]^{n \times 1}$ (see, e.g., \cite{cz-bregman,scoring-rules}). At this aim, first $s$ is extended to a function on $[0,1] \times [0,1]$ setting, for every $p,x \in [0,1]$
$$
s(p,x) = ps(1,x)+(1-p)s(0,x).
$$
Then we define (see \cite{gs-sanfilippo}), for every ${\bf u}, {\bf v} \in [0,1]^{n \times 1}$,
$$
d_{s}({\bf u},{\bf v}) = \sum_{i=1}^n s(u_i,v_i) - \sum_{i=1}^n s(u_i,u_i).
$$

In particular, the Bregman divergence corresponding to the Brier quadratic scoring rule $d_{s_B}$ coincides with the squared Euclidean distance, i.e., $d_{s_B}({\bf u},{\bf v}) = d_E({\bf u},{\bf v})^2$.

The notions of coherence-5P and coherence-5N generalize accordingly to {\it coherence-5P($s$)} and {\it coherence-5N($s$)}, substituting the Brier quadratic scoring rule with a general proper scoring rule.

\begin{theorem}
Let $s$ be a fixed proper scoring rule.
The following statements hold:
\begin{description}
\item[\it (i)] a belief assessment $Bel:\E \to [0,1]$ is coherent-3 (and so coherent) if and only if it is coherent-5($s$);
\item[\it (ii)] a probability assessment $P:\E \to [0,1]$ is coherent-3P (and so coherent-P) if and only if it is coherent-5P($s$);
\item[\it (iii)] a necessity assessment $N:\E \to [0,1]$ is coherent-3N (and so coherent-N) if and only if it is coherent-5N($s$).
\end{description}
\end{theorem}
\begin{proof}
We only focus on statement {\it (i)} since the other two statements can be proved analogously.
For $n \in \NN$ and $\F = \{E_1,\ldots,E_n\} \subseteq \E$, the equivalence between coherence-3 and coherence-5($s$) follows by Theorem~1 and Corollary~1 in \cite{scoring-rules}.
Indeed, by our Theorem~\ref{th:cohe-bel}, coherence-3 is equivalent to the existence of a belief function $Bel''$ on $\A_\F$ extending the restriction of $Bel$ to $\F$. Such $Bel''$ is completely characterized by its M\"obius inverse $m_{Bel''}$ on $\A_\F$ that gives rise to a probability measure $\mu$ on $\wp(\U_\F)$. Defining $\tilde{E}_i^{\F} = \{B \in \U_\F \,:\, B \subseteq E_i\}$, for $i=1,\ldots,n$, we have that $Bel(E_i) = \mu(\tilde{E}_i^{\F})$. 
So, the coherence of the belief assessment $Bel(E_1),\ldots,Bel(E_n)$ is equivalent to the coherence (in the classical de Finetti's sense) of the probability assessment
$\mu(\tilde{E}_1^\F),\ldots,\mu(\tilde{E}_n^\F)$. Further, for every $\G \in \wp(\U_\F)$, denoting by $\chi_\G:\U_\F \to \{0,1\}$ the corresponding indicator, we have that
$$
\chi_{\tilde{E}_i^\F}(B) = \min_{\omega\in B} \1_{E_i}(\omega).
$$
Hence, setting $\mu^*(\tilde{E}_i^\F) = Bel^*(E_i)$, for $i=1,\ldots,n$, the functions $L^*_\F,L_\F$ in condition coherence-5($s$) reduce to
$$
L_\F^*(B) = \sum_{i=1}^n s(\chi_{\tilde{E}_i^\F}(B),\mu^*(\tilde{E}_i^\F))
\quad
\mbox{and}
\quad
L_\F(B) = \sum_{i=1}^n s(\chi_{\tilde{E}_i^\F}(B),\mu(\tilde{E}_i^\F)).
$$
Finally, by Theorem~1 and Corollary~1 in \cite{scoring-rules}, the probability assessment $\mu(\tilde{E}_1^\F),\ldots,\mu(\tilde{E}_n^\F)$ is coherent if and only if there is no distinct assessment $\mu^*(\tilde{E}_1^\F),\ldots,\mu^*(\tilde{E}_n^\F)$ such that $L_\F^*(B) < L_\F(B)$, for every $B \in \U_\F$.
\hfill$\square$
\end{proof}

An assessment $Bel:\E \to [0,1]$ can be regarded either inside the wider framework of coherent belief assessments or inside the narrower frameworks of coherent-P or coherent-N assessments. We have that incoherence implies both incoherence-P and incoherence-N, but incoherence-P does not imply incoherence as well as incoherence-N does not imply incoherence.

Focusing on a finite $\E = \{E_1,\ldots,E_n\}$, condition coherence-5($s$) (respectively, coherence-5P($s$) and coherence-5N($s$)) can be used to correct an incoherent (respectively, incoherent-5P($s$) and incoherent-5N($s$)) assessment.
Let $\A_\E = \algebra(\E)$, $\U_\E = \A_\E \setminus \{\emptyset\}$ and $\C_\E = \atoms(\A_\E) = \{C_1,\ldots,C_m\}$. Fix an enumeration of $\U_\E = \{A_1,\ldots,A_{2^m - 1}\}$, let ${\bf e}^1,\ldots,{\bf e}^{2^m-1} \in \RR^{n \times 1}$ be defined as in the proof of Theorem~\ref{th:cohe-bel}, moreover, let ${\bf d} = (Bel(E_1),\ldots,Bel(E_n))^T \in \RR^{n \times 1}$.

As proved in Theorem~\ref{th:cohe-bel}, $Bel$ is incoherent if and only if ${\bf d}$ does not belong to the convex hull of ${\bf e}^1,\ldots,{\bf e}^{2^m-1}$, in symbol ${\bf d} \notin \conv(\{{\bf e}^1,\ldots,{\bf e}^{2^m-1}\})$. Hence, fixed a Bregman divergence $d_s$ determined by a bounded proper scoring rule $s$, by Proposition~3 in \cite{scoring-rules} we can correct the assessment finding the projection of ${\bf d}$ onto $\conv(\{{\bf e}^1,\ldots,{\bf e}^{2^m-1}\})$, that is
$${\bf d}^* = \displaystyle{\arg\min_{{\bf u} \in \conv(\{{\bf e}^1,\ldots,{\bf e}^{2^m-1}\})} d_s({\bf u},{\bf d})}.
$$

The assessment is incoherent-P if and only if ${\bf d}$ does not belong to the convex hull of those ${\bf e}^j$'s corresponding to elements of $\C_\E$. If it is so, then the correction inside the subclass of coherent-P probability assessments ${\bf d}^{**}$ is found as the projection of ${\bf d}$ onto the convex hull of those ${\bf e}^j$'s corresponding to elements of $\C_\E$.

Finally, the assessment is incoherent-N if and only if ${\bf d}$ does not belong to the convex hull of those ${\bf e}^j$'s corresponding to elements of $\D_\E$, for any $\D_\E \in \chains(\U_\E)$. In this case, to find the correction inside the subclass of coherent-N necessity assessments, for each $\D_\E \in \chains(\U_\E)$ we need to look for the projection of ${\bf d}$ onto the convex hull of those ${\bf e}^j$'s corresponding to elements of $\D_\E$, say ${\bf d}^{***}_{\D_\E}$, and then select the final correction as 
$$
{\bf d}^{***} \in \arg\min_{{\bf d}^{***}_{\D_\E},\D_\E \in \chains(\U_\E)} d_s({\bf d}^{***}_{\D_\E}, {\bf d}).
$$

\begin{example}
Take $\Omega = \{\omega_1,\omega_2,\omega_3\}$, $A = \{\omega_1,\omega_2\}$ and $B = \{\omega_2,\omega_3\}$, together with the belief assessment $Bel(A) = \frac{1}{4}$, $Bel(B) = 1$ and $Bel(A \cap B) = \frac{1}{2}$. Such assessment is immediately seen to be incoherent (and so also incoherent-P and incoherent-N), since $A \cap B \subseteq A$ but $Bel(A \cap B) = \frac{1}{2} > \frac{1}{4} = Bel(A)$.

In particular, for $\E = \{A,B,A\cap B\}$, since $\A_\E = \wp(\Omega)$, fixing the enumeration of $\U_\E = \{A_1,\ldots,A_7\}$ with
$A_1 = \{\omega_1\}$, $A_2 = \{\omega_2\}$, $A_3 = \{\omega_3\}$, $A_4 = \{\omega_1,\omega_2\}$, $A_5 = \{\omega_1,\omega_3\}$, $A_6 = \{\omega_2,\omega_3\}$ and
$A_7 = \Omega$, we have
\begin{center}
\begin{tabular}{c|ccccccc}
$\U_\E$ & $A_1$ & $A_2$ & $A_3$ & $A_4$ & $A_5$ & $A_6$ & $A_7$\\
\hline
$\min \1_A$ & $1$ & $1$ & $0$ & $1$ & $0$ & $0$ & $0$\\
$\min \1_B$ & $0$ & $1$ & $1$ & $0$ & $0$ & $1$ & $0$\\
$\min \1_{A\cap B}$ & $0$ & $1$ & $0$ & $0$ & $0$ & $0$ & $0$
\end{tabular}
\end{center}
whose columns are the vectors ${\bf e}^1,\ldots,{\bf e}^7$.

Denoting ${\bf d} = (Bel(A),Bel(B),Bel(A \cap B))^T$, the incoherence of the assessment above is due to the fact that ${\bf d} \notin \conv(\{{\bf e}^1,\ldots,{\bf e}^7\})$. In particular, since ${\bf e}^1 = {\bf e}^4$, ${\bf e}^3 = {\bf e}^6$ and ${\bf e}^5 = {\bf e}^7$, it holds $\conv(\{{\bf e}^1,\ldots,{\bf e}^7\}) = \conv(\{{\bf a}^1,{\bf a}^2,{\bf a}^3,{\bf a}^4\})$ where ${\bf a}^1 = \alpha {\bf e}^1 + (1-\alpha) {\bf e}^4$, ${\bf a}^2 = {\bf e}^2$,
${\bf a}^3 = \beta {\bf e}^3 + (1-\beta) {\bf e}^6$,
${\bf a}^4 = \gamma {\bf e}^5 + (1-\gamma) {\bf e}^7$, with $\alpha,\beta,\gamma \in [0,1]$. Using the Bregman divergence related to the Brier quadratic scoring rule denoted as $d_{s_B}$, the projection of ${\bf d}$ onto $\conv(\{{\bf a}^1,{\bf a}^2,{\bf a}^3,{\bf a}^4\})$ is
\begin{eqnarray*}
{\bf d}^* &=& 0\cdot {\bf a}^1 + \frac{3}{8} \cdot {\bf a}^2 + \frac{5}{8} \cdot {\bf a}^3 + 0 \cdot {\bf a}^4\\
&=&
0 \cdot {\bf e}^1 + \frac{3}{8} \cdot {\bf e}^2 + \lambda \cdot {\bf e}^3 + 0 \cdot {\bf e}^4 + 0 \cdot {\bf e}^5 + \left(\frac{5}{8} - \lambda\right) \cdot {\bf e}^6 +  0 \cdot {\bf e}^7,
\end{eqnarray*}
where $\lambda \in \left[0,\frac{5}{8}\right]$. This implies that there is an entire class of belief functions on $\A_\E$ providing the same correction of the assessment on $\E$, whose M\"obius inverse on $\U_\E$ is given by the weights of the last convex combination. Explicitly, we have
\begin{center}
\begin{tabular}{c|cccccccc}
$\A_\E$ & $\emptyset$ & $\{\omega_1\}$ & $\{\omega_2\}$ & $\{\omega_3\}$ & $\{\omega_1,\omega_2\}$ & $\{\omega_1,\omega_3\}$ & $\{\omega_2,\omega_3\}$ & $\Omega$\\
\hline
$m_{Bel^*}$ & $0$ & $0$ & $\frac{3}{8}$ & $\lambda$ & $0$ & $0$ & $\frac{5}{8}-\lambda$ & $0$\\
$Bel^*$ & $0$ & $0$ & $\frac{3}{8}$ & $\lambda$ & $\frac{3}{8}$ & $\lambda$ & $1$ & $1$\\
\end{tabular}
\end{center}
so, independently of $\lambda \in \left[0,\frac{5}{8}\right]$, we get the correction $Bel^*(A) = \frac{3}{8}$, $Bel^*(B) = 1$ and $Bel^*(A \cap B) = \frac{3}{8}$, for which $d_{s_B}({\bf d}^*,{\bf d}) = \frac{1}{32}$.

If we consider the incoherent-P probability assessment $P = Bel$, then we can correct it by referring to the subclass of coherent-P probability assessments. Noticing that the vectors ${\bf e}^1,{\bf e}^2,{\bf e}^3$ correspond to $\atoms(\A_\E)$, this can be done looking for the projection of ${\bf d}$ onto $\conv(\{{\bf e}^1,{\bf e}^2,{\bf e}^3\})$ which is
$$
{\bf d}^{**} = 0 \cdot {\bf e}^1 + \frac{3}{8} \cdot {\bf e}^2 + \frac{5}{8} \cdot {\bf e}^3.
$$
The weights of such convex combination are the restriction to $\atoms(\A_\E)$ of a probability measure $P^*$ on $\A_\E$ correcting the assessment $P$. Such $P^*$ reveals to be a particular element of the class of belief functions correcting the initial assessment, obtained for $\lambda = \frac{5}{8}$.

Finally, if we consider the incoherent-N necessity assessment $N = Bel$, then we can correct it by referring to the subclass of coherent-N necessity assessments. To do so, we need to consider all the permutations of $\atoms(\A_\E)$, each of them giving rise to an element of $\chains(\U_\E)$. For each $\D_\E \in \chains(\U_\E)$ we need to look for the projection of ${\bf d}$ onto the convex hull of those ${\bf e}^j$'s corresponding to elements in $\D_\E$, say ${\bf d}^{***}_{\D_\E}$, and then select the final correction as
$$
{\bf d}^{***} \in \arg\min_{{\bf d}^{***}_{\D_\E},\D_\E \in \chains(\U_\E)} d_{s_B}({\bf d}^{***}_{\D_\E}, {\bf d}).
$$
We have that the minimum is achieved in correspondence of
$\D_\E = \{A_2,A_6,A_7\}$ and it holds
$$
{\bf d}^{***} = \frac{3}{8} \cdot {\bf e}^2 + \frac{5}{8} \cdot {\bf e}^6 + 0 \cdot {\bf e}^7,
$$
whose coefficients give rise (by setting to $0$ the missing weights) to a M\"obius inverse on $\A_\E$, which is in bijection with a necessity measure $N^*$ on $\A_\E$ correcting the assessment $N$. Such $N^*$ reveals to be a particular element of the class of belief functions correcting the initial assessment, obtained for $\lambda = 0$.

Let us stress that the correction we get depends on the particular Bregman divergence (that is on the particular bounded proper scoring rule) we choose.
\hfill$\blacklozenge$
\end{example}

In the literature, there are several studies using divergences and distances with the purpose of correcting an incoherent assessment or approximating a coherent assessment with another belonging to a different framework.
For instance, concerning correction, in \cite{capot} the authors propose a discrepancy measure for the correction of an incoherent conditional probability assessment. Such measure is linked to the logarithmic Bregman divergence (see \cite{gs-sanfilippo}) determined by the logarithmic scoring rule $s_L$, which is an unbounded proper scoring rule. While, concerning approximation, in \cite{mmv} the authors cope with the outer approximation of a coherent lower probability with a belief function. Such task is achieved by minimizing some suitable distances defined on the set of lower probabilities.

\section{Conclusions}
This contribution aims at celebrating our \red{esteemed} colleague Teddy Seidenfeld in the occasion of his 70th birthday.
We pay homage to him by presenting new results connected to some topics he faced in his brilliant career, in particular, scoring rules \cite{ssk-scoring,seidenfeld,ssk-imprecise} and belief functions \cite{seidenfeld2,ws-jspi}.

We introduce various notions of coherence for a partial assessment in the Dempster-Shafer theory of evidence. In detail, we present a betting scheme condition and a penalty criterion condition that mimic de Finetti's conditions for probabilistic assessments \cite{definetti}. The latter two conditions rely on the partially resolving uncertainty principle introduced by Jaffray \cite{Jaffray-Bel}. Such conditions can be considered as operational tools to assess subjective beliefs.
We prove that all the introduced notions of coherence are equivalent to the consistency with a belief function defined on an algebra.

Then we focus on two distinguished subclasses of belief functions, given by finitely additive probability measures and finitely minitive necessity measures. For both cases, we specialize the introduced coherence notions and equivalence results.

Finally, in each of the three analyzed frameworks, we provide a generalized coherence condition based on a (strictly) proper scoring rule, possibly departing from the classical Brier quadratic scoring rule. We show the equivalence of such condition with the other coherence conditions and we face the problem of correcting an incoherent assessment using the Bregman divergence determined by a bounded proper scoring rule.


\begin{acknowledgements}
The authors are members of the INdAM-GNAMPA research group. The second author was supported by Università degli Studi di Perugia, Fondo Ricerca di Base 2019, project ``Modelli per le decisioni economiche e finanziarie in condizioni di ambiguità ed imprecisione''.
\end{acknowledgements}

%
%

\bibliographystyle{plain}
\bibliography{biblio.bib}   

\begin{thebibliography}{10}

\bibitem{bhaskara}
K.P.S. {Bhaskara Rao} and M.~{Bhaskara Rao}.
\newblock {\em {Theory of Charges: A Study of Finitely Additive Measures}}.
\newblock Academic Press, 1983.

\bibitem{Brier}
G.W. Brier.
\newblock Verification of forecasts expressed in terms of probability.
\newblock {\em Monthly Weather Review}, 78(1):1--3, 1950.

\bibitem{ccv-kyb2014}
A.~Capotorti, G.~Coletti, and B.~Vantaggi.
\newblock {Standard and nonstandard representability of positive uncertainty
  orderings}.
\newblock {\em Kybernetika}, 50(2):189--215, 2014.

\bibitem{capot}
A.~Capotorti, G.~Regoli, and F.~Vattari.
\newblock Correction of incoherent conditional probability assessments.
\newblock {\em International Journal of Approximate Reasoning}, 51(6):718--727,
  2010.

\bibitem{cz-bregman}
Y.~Censor and S.A. Zenios.
\newblock {\em Parallel Optimization: Theory, Algorithms and Applications}.
\newblock Oxford University Press, 1997.

\bibitem{cj-2mon}
A.~Chateauneuf and J.-Y. Jaffray.
\newblock {Some characterizations of lower probabilities and other monotone
  capacities through the use of M\"obius inversion}.
\newblock {\em Mathematical Social Sciences}, 17(3):263--283, 1989.

\bibitem{choquet}
G.~Choquet.
\newblock Theory of capacities.
\newblock {\em Annales de l'Institut Fourier}, 5:131--295, 1953.

\bibitem{cpv-sma}
G.~Coletti, D.~Petturiti, and B.~Vantaggi.
\newblock Bayesian inference: the role of coherence to deal with a prior belief
  function.
\newblock {\em Statistical Methods {\&} Applications}, 23(4):519--545, 2014.

\bibitem{cpv-kybernetika}
G.~Coletti, D.~Petturiti, and B.~Vantaggi.
\newblock Rationality principles for preferences on belief functions.
\newblock {\em Kybernetika}, 51(3):486--507, 2015.

\bibitem{cpv-is2016}
G.~Coletti, D.~Petturiti, and B.~Vantaggi.
\newblock Conditional belief functions as lower envelopes of conditional
  probabilities in a finite setting.
\newblock {\em Information Sciences}, 339:64--84, 2016.

\bibitem{cpv-fss2016}
G.~Coletti, D.~Petturiti, and B.~Vantaggi.
\newblock When upper conditional probabilities are conditional possibility
  measures.
\newblock {\em Fuzzy Sets and Systems}, 304:45--64, 2016.

\bibitem{cpv-anor2019}
G.~Coletti, D.~Petturiti, and B.~Vantaggi.
\newblock Dutch book rationality conditions for conditional preferences under
  ambiguity.
\newblock {\em Annals of Operations Research}, 279(1):115--150, 2019.

\bibitem{cpv-ijar2019}
G.~Coletti, D.~Petturiti, and B.~Vantaggi.
\newblock Models for pessimistic or optimistic decisions under different
  uncertain scenarios.
\newblock {\em International Journal of Approximate Reasoning}, 105:305--326,
  2019.

\bibitem{bumi}
G.~Coletti, D.~Petturiti, and B.~Vantaggi.
\newblock {A Dutch book coherence condition for conditional completely
  alternating Choquet expectations}.
\newblock {\em Bollettino dell'Unione Matematica Italiana}, 13(4):585--593,
  2020.

\bibitem{cs-libro}
G.~Coletti and R.~Scozzafava.
\newblock {\em Probabilistic Logic in a Coherent Setting}, volume~15 of {\em
  Trends in Logic}.
\newblock Kluwer Academic Publisher, Dordrecht/Boston/London, 2002.

\bibitem{cs-bel}
G.~Coletti and R.~Scozzafava.
\newblock Toward a general theory of conditional beliefs.
\newblock {\em International Journal of Intelligent Systems}, 21:229--259,
  2006.

\bibitem{csv-is}
G.~Coletti, R.~Scozzafava, and B.~Vantaggi.
\newblock Inferential processes leading to possibility and necessity.
\newblock {\em Informtion Sciences}, 245(1):132--145, 2013.

\bibitem{cv-ijar2007}
G.~Coletti and B.~Vantaggi.
\newblock {Comparative models ruled by possibility and necessity: A conditional
  world}.
\newblock {\em International Journal of Approximate Reasoning}, 45(2):341--363,
  2007.

\bibitem{cv2008}
G.~Coletti and B.~Vantaggi.
\newblock A view on conditional measures through local representability of
  binary relations.
\newblock {\em International Journal of Approximate Reasoning}, 47(1):268--283,
  2008.

\bibitem{dCtm}
G.~{de Cooman}, M.C.M. Troffaes, and E.~Miranda.
\newblock {n-Monotone exact functionals}.
\newblock {\em Journal of Mathematical Analysis and Applications},
  347(1):143--156, 2008.

\bibitem{definetti}
B.~{de Finetti}.
\newblock {\em Theory of Probability 1-2}.
\newblock John Wiley \& Sons, London, New York, Sydney, Toronto, 1975.

\bibitem{dempster}
A.P. Dempster.
\newblock {Upper and Lower Probabilities Induced by a Multivalued Mapping}.
\newblock {\em Annals of Mathematical Statistics}, 38(2):325--339, 1967.

\bibitem{dubprad}
D.~Dubois and H.~Prade.
\newblock {\em Possibility Theory: An Approach to Computerized Processing of
  Uncertainty}.
\newblock Plenum Press, New York and London, 1988.

\bibitem{dubprad-when}
D.~Dubois and H.~Prade.
\newblock When upper probabilities are possibility measures.
\newblock {\em Fuzzy Sets and Systems}, 49(1):65--74, 1992.

\bibitem{gs-mobius}
I.~Gilboa and D.~Schmeidler.
\newblock Canonical representation of set functions.
\newblock {\em Mathematics of Operations Research}, 20(1):197--212, 1995.

\bibitem{gilio}
A.~Gilio.
\newblock Criterio di penalizzazione e condizioni di coerenza nella valutazione
  soggettiva della probabilit\`a.
\newblock {\em Bollettino U.M.I.}, 4-B(3):645--660, 1990.

\bibitem{gs-sanfilippo}
A.~Gilio and G.~Sanfilippo.
\newblock {Coherent conditional probabilities and proper scoring rules}.
\newblock In {\em {Proceedings of ISIPTA 2011}}. 2011.

\bibitem{Good}
I.J. Good.
\newblock Rational decisions.
\newblock {\em Journal of the Royal Statistical Society. Series B
  (Methodological)}, 14(1):107--114, 1952.

\bibitem{grabisch}
M.~Grabisch.
\newblock {\em Set Functions, Games and Capacities in Decision Making}.
\newblock Springer, 2016.

\bibitem{halpern}
J.~Halpern.
\newblock {\em Reasoning about uncertainty}.
\newblock The MIT Press, 2003.

\bibitem{Jaffray-Bel}
J.-Y. Jaffray.
\newblock Coherent bets under partially resolving uncertainty and belief
  functions.
\newblock {\em Theory and Decision}, 26(2):99--105, 1989.

\bibitem{KerkMeest}
T.~Kerkvliet and R.~Meester.
\newblock {A Behavioral Interpretation of Belief Functions}.
\newblock {\em Journal of Theoretical Probability}, 31(4):2112--2128, 2018.

\bibitem{krat}
V.~Kr\"atschmer.
\newblock When fuzzy measures are upper envelopes of probability measures.
\newblock {\em Fuzzy Sets and Systems}, 138(3):455--468, 2003.

\bibitem{mangasarian}
O.L. Mangasarian.
\newblock {\em Nonlinear Programming}, volume~10 of {\em Classics in Applied
  Mathematics}.
\newblock SIAM, 1994.

\bibitem{m-mobius}
M.~Marinacci.
\newblock Decomposition and representation of coalitional games.
\newblock {\em Mathematics of Operations Research}, 21(4):1000--1015, 1996.

\bibitem{mmv}
I.~Montes, E.~Miranda, and P.~Vicig.
\newblock Outer approximating coherent lower probabilities with belief
  functions.
\newblock {\em International Journal of Approximate Reasoning}, 110:1--30,
  2019.

\bibitem{nguyen}
H.T. Nguyen.
\newblock {\em An Introduction to Random Sets}.
\newblock Chapman \& Hall/CRC, 2006.

\bibitem{pv-jmaa2018}
D.~Petturiti and B.~Vantaggi.
\newblock {Upper and lower conditional probabilities induced by a multivalued
  mapping}.
\newblock {\em Journal of Mathematical Analysis and Applications},
  458(2):1214--1235, 2018.

\bibitem{scoring-rules}
J.B. Predd, R.~Seiringer, E.H. Lieb, D.N. Osherson, H.V. Poor, and S.R.
  Kulkarni.
\newblock {Probabilistic Coherence and Proper Scoring Rules}.
\newblock {\em IEEE Transactions on Information Theory}, 55(10):4786--4792,
  2009.

\bibitem{Regoli1994}
G.~Regoli.
\newblock {Rational Comparisons and Numerical Representations}.
\newblock In S.~R{\'i}os, editor, {\em {Decision Theory and Decision Analysis:
  Trends and Challenges}}, pages 113--126. Springer Netherlands, Dordrecht,
  1994.

\bibitem{ssk-scoring}
M.J. Schervish, T.~Seidenfeld, and J.B. Kadane.
\newblock {Proper Scoring Rules, Dominated Forecasts, and Coherence}.
\newblock {\em Decision Analysis}, 6(4):202--221, 2009.

\bibitem{schmeidler2}
D.~Schmeidler.
\newblock Integral representation without additivity.
\newblock {\em Proceedings of the American Mathematical Society},
  97(2):255--261, 1986.

\bibitem{seidenfeld2}
T.~Seidenfeld.
\newblock Statistical evidence and belief functions.
\newblock {\em PSA: Proceedings of the Biennial Meeting of the Philosophy of
  Science Association}, 1978:478--489, 1978.

\bibitem{seidenfeld}
T.~Seidenfeld.
\newblock Calibration, coherence, and scoring rules.
\newblock {\em Philosophy of Science}, 52(2):274--294, 1985.

\bibitem{ssk-imprecise}
T.~Seidenfeld, M.J. Schervish, and J.B. Kadane.
\newblock {Forecasting with imprecise probabilities}.
\newblock {\em International Journal of Approximate Reasoning},
  458(2):1214--1235, 2018.

\bibitem{shafer}
G.~Shafer.
\newblock {\em A Mathematical Theory of Evidence}.
\newblock Princeton University Press, Princeton, NJ, 1976.

\bibitem{deCooman-book}
M.C.M. Troffaes and G.~{de~Cooman}.
\newblock {\em Lower Previsions}.
\newblock Wiley Series in Probability and Statistics. Wiley, 2014.

\bibitem{walley-lp}
P.~Walley.
\newblock Coherent lower (and upper) probabilities.
\newblock Technical Report~22, Department of Statistics, University of Warwick,
  UK, 1981.

\bibitem{walley-libro}
P.~Walley.
\newblock {\em Statistical Reasoning with Imprecise Probabilities}.
\newblock Chapman and Hall, London, 1991.

\bibitem{ws-jspi}
L.~Wasserman and T.~Seidenfeld.
\newblock {The dilation phenomenon in robust Bayesian inference}.
\newblock {\em Journal of Statistical Planning and Inference}, 40(2):345--356,
  1994.

\bibitem{wasserman}
L.A. Wasserman.
\newblock {Prior Envelopes Based on Belief Functions}.
\newblock {\em The Annals of Statistics}, 18(1):454--464, 1990.

\bibitem{williams}
P.M. Williams.
\newblock Note on conditional previsions.
\newblock Unpublished report of School of Mathematical and Physical Science,
  University of Sussex (Published in {\em International Journal of Approximate
  Reasoning}, 44:366--383, 2007), 1975.

\end{thebibliography}


\end{document}